\documentclass[]{amsart}

\usepackage[english]{babel}
\usepackage[utf8]{inputenc}
\usepackage[T1]{fontenc}
\usepackage{microtype}
\usepackage{paralist}
\usepackage{amsmath, amsfonts, amssymb, amsthm}
\usepackage{color}
\usepackage[normalem]{ulem}
\usepackage{tikz}

\newcommand{\CC}{\mathbb{C}}

\newcommand{\NN}{\mathbb{N}}

\newcommand{\ZZ}{\mathbb{Z}}

\newcommand{\Hc}{\mathcal{H}}

\newcommand{\Oc}{\mathcal{O}}

\newcommand{\gfr}{\mathfrak{g}}
\newcommand{\set}[1]{\left\{ #1 \right\}}
\newcommand{\setb}[1]{\left( #1 \right)}
\newcommand{\abs}[1]{\left| #1 \right|}

\newcommand{\bino}[2]{\begin{pmatrix} #1 \\ #2 \end{pmatrix}}

\newtheorem{mymasterthm}{notForUse}
\theoremstyle{definition}

\newtheorem{myrem}[mymasterthm]{Remark}

\theoremstyle{plain}
\newtheorem{mylemma}[mymasterthm]{Lemma}
\newtheorem{mythm}[mymasterthm]{Theorem}

\allowdisplaybreaks

\title[$ S $-unit values of $ G_n + G_m $ in function fields]{$ S $-unit values of $ G_n + G_m $ in function fields}
\makeatletter
\@namedef{subjclassname@2020}{\textup{2020} Mathematics Subject Classification}
\makeatother

\author[S. Heintze]{Sebastian Heintze}
\address{Sebastian Heintze\newline
	\indent Graz University of Technology\newline
	\indent Institute of Analysis and Number Theory\newline
	\indent Steyrergasse 30/II \newline
	\indent A-8010 Graz, Austria}
\email{heintze@math.tugraz.at}

\thanks{Supported by Austrian Science Fund (FWF) under project I4406 and I4945-N}

\begin{document}
	
	\maketitle
	
	
	\begin{abstract}
		In this paper we consider a simple linear recurrence sequence $ G_n $ defined over a function field in one variable over the field of complex numbers.
		We prove an upper bound on the indices $ n $ and $ m $ such that $ G_n + G_m $ is an $ S $-unit.
		This is a function field analogue of already known results in number fields.
	\end{abstract}
	
	\section{Introduction}
	
	By a linear recurrence sequence we mean a sequence $ (G_n)_{n \in \NN} $ given by its Binet representation
	\begin{equation*}
		G_n = f_1 \alpha_1^n + \cdots + f_d \alpha_d^n.
	\end{equation*}
	It is well known that such a sequence satisfies a linear recurrence relation.
	Without loss of generality we may assume that the characteristic roots $ \alpha_i $ are pairwise distinct and that all $ \alpha_i $ and $ f_i $ are non-zero.
	In general the $ f_i $ are polynomials in $ n $, but we will restrict ourselves to the case when the coefficients $ f_i $ are constant (i.e. independent of $ n $).
	Such a linear recurrence sequence is called simple.
	If $ G_n $ is defined over a number field, then it is called non-degenerate if no quotient $ \alpha_i / \alpha_j $ for $ i \neq j $ is a root of unity, and if $ G_n $ is defined over a function field, then it is called non-degenerate if no quotient $ \alpha_i / \alpha_j $ for $ i \neq j $ is constant.
	
	In the number field case the behaviour of the sum of two elements of such a linear recurrence sequence was already studied by several authors.
	Bravo and Luca considered in \cite{bravo-luca-2016} the Diophantine equation
	\begin{equation*}
		F_n + F_m = 2^a
	\end{equation*}
	where $ (F_n)_{n \in \NN} $ denotes the sequence of Fibonacci numbers.
	This result was then generalized by Pink and Ziegler \cite{pink-ziegler-2018} to the more general equation
	\begin{equation*}
		G_n + G_m = w p_1^{z_1} \cdots p_s^{z_s}
	\end{equation*}
	for a binary non-degenerate linear recurrence sequence $ (G_n)_{n \in \NN} $ and primes $ p_1, \ldots, p_s $.
	In other words, they studied the situation when the sum $ G_n + G_m $ is a so-called $ S $-unit, i.e. has only primes of $ S $ in its prime factorization.
	Moreover, the case that $ G_n + G_m $ is a sum of $ S $-units was recently studied by Bhoi et al. in \cite{bhoi-panda-rout-}.
	
	The purpose of the present paper is to study the analogous situation in a function field setting (of characteristic zero).
	Namely, we are interested in bounding the indices of a linear recurrence sequence, defined over a function field in one variable over the field of complex numbers, if the sum of some elements of this sequence is an $ S $-unit.
	
	\section{Notation and results}
	
	Throughout the whole paper we denote by $ F $ a function field in one variable over $ \CC $ and by $ \gfr $ the genus of $ F $.
	We will work with valuations and give here for the readers convenience a short wrap-up of this notion that can e.g. also be found in \cite{fuchs-heintze-p8}:
	For $ c \in \CC $ and $ f(x) \in \CC(x) $, where $ \CC(x) $ is the rational function field over $ \CC $, we denote by $ \nu_c(f) $ the unique integer such that $ f(x) = (x-c)^{\nu_c(f)} p(x) / q(x) $ with $ p(x),q(x) \in \CC[x] $ such that $ p(c)q(c) \neq 0 $. Further we write $ \nu_{\infty}(f) = \deg q - \deg p $ if $ f(x) = p(x) / q(x) $.
	These functions $ \nu : \CC(x) \rightarrow \ZZ $ are up to equivalence all valuations in $ \CC(x) $.
	If $ \nu_c(f) > 0 $, then $ c $ is called a zero of $ f $, and if $ \nu_c(f) < 0 $, then $ c $ is called a pole of $ f $, where $ c \in \CC \cup \set{\infty} $.
	For a finite extension $ F $ of $ \CC(x) $ each valuation in $ \CC(x) $ can be extended to no more than $ [F : \CC(x)] $ valuations in $ F $. This again gives up to equivalence all valuations in $ F $.
	Both, in $ \CC(x) $ as well as in $ F $ the sum-formula
	\begin{equation*}
		\sum_{\nu} \nu(f) = 0
	\end{equation*}
	holds, where the sum is taken over all valuations (up to equivalence) in the considered function field.
	For a finite set $ S $ of valuations on $ F $, we denote by $ \Oc_S^* $ the set of $ S $-units in $ F $, i.e. the set
	\begin{equation*}
		\Oc_S^* = \set{f \in F^* : \nu(f) = 0 \text{ for all } \nu \notin S}.
	\end{equation*}
	Finally, we call two elements $ \alpha, \beta \in F $ multiplicatively independent if $ \alpha^r \beta^s \in \CC $ for $ r,s \in \ZZ $ implies that $ r = s = 0 $.
	
	Our first theorem is the following finiteness result which states that there are only finitely many $ S $-units in a linear recurrences sequence of order at least two:
	
	\begin{mythm}
		\label{thm:onesummand}
		Let $ G_n = f_1 \alpha_1^n + \cdots + f_d \alpha_d^n $ be a simple linear recurrence sequence of order $ d \geq 2 $, defined over $ F $.
		Assume that $ G_n $ is non-degenerate.
		Furthermore, let $ S $ be a finite set of valuations on $ F $.
		Then there exists an effectively computable constant $ C $ such that
		\begin{equation*}
			G_n \in \Oc_S^*
		\end{equation*}
		implies $ n \leq C $.
	\end{mythm}
	
	In our second theorem we consider now the sum of two elements of a linear recurrence sequence:
	
	\begin{mythm}
		\label{thm:twosummands}
		Let $ G_n = f_1 \alpha_1^n + \cdots + f_d \alpha_d^n $ be a simple linear recurrence sequence of order $ d \geq 2 $, defined over $ F $.
		Assume that $ G_n $ is non-degenerate, that no $ \alpha_i $ lies in $ \CC $, and that for all $ i \neq j $ the characteristic roots $ \alpha_i $ and $ \alpha_j $ are multiplicatively independent.
		Furthermore, let $ S $ be a finite set of valuations on $ F $.
		Then there exists an effectively computable constant $ C $ such that
		\begin{equation*}
			G_n + G_m \in \Oc_S^*
		\end{equation*}
		for $ n > m $ implies $ \max \setb{n,m} \leq C $.
	\end{mythm}
	
	Before we start with preparing the proofs of our theorems let us give a remark on possible generalizations of our results:
	
	\begin{myrem}
		The proof of Theorem \ref{thm:twosummands} given below can be adapted to the case when we consider the analogous situation for more summands.
		More precisely, under the conditions of Theorem \ref{thm:twosummands}, we get for
		\begin{equation}
			\label{eq:generalized}
			G_{n_1} + \cdots + G_{n_t} \in \Oc_S^*
		\end{equation}
		with $ n_1 > \cdots > n_t $ an effectively computable upper bound $ \max_i n_i \leq C $.
		The proof has the same strategy as the proof of Theorem \ref{thm:twosummands}.
		Considering the related $ S $-unit equation, one first bounds the $ n_i $ appearing twice in a minimal vanishing subsum.
		In the next step those $ n_i $ which are in the same minimal vanishing subsum as an already bounded one were bounded.
		Finally, one has to consider the situation that some of the exponents only appear in minimal vanishing subsums with different exponents that are not bounded so far.
		Here the case if the same characteristic root occurs twice can be reduced to an equation of type \eqref{eq:generalized} with less summands (some kind of induction), and the case of different appearing characteristic roots uses the multiplicative independence assumption.
		We leave it up to the interested reader to write this down in detail.
	\end{myrem}
	
	\section{Preliminaries}
	
	The proofs in the next section will use height functions in function fields. Hence, let us define the height of an element $ f \in F^* $ by
	\begin{equation*}
		\Hc(f) := - \sum_{\nu} \min \setb{0, \nu(f)} = \sum_{\nu} \max \setb{0, \nu(f)}
	\end{equation*}
	where the sum is taken over all valuations (up to equivalence) on the function field $ F / \CC $. Additionally we define $ \Hc(0) = \infty $.
	This height function satisfies some basic properties, listed in the lemma below which is proven in \cite{fuchs-karolus-kreso-2019}:
	
	\begin{mylemma}
		\label{lemma:heightproperties}
		Denote as above by $ \Hc $ the height on $ F/\CC $. Then for $ f,g \in F^* $ the following properties hold:
		\begin{enumerate}[a)]
			\item $ \Hc(f) \geq 0 $ and $ \Hc(f) = \Hc(1/f) $,
			\item $ \Hc(f) - \Hc(g) \leq \Hc(f+g) \leq \Hc(f) + \Hc(g) $,
			\item $ \Hc(f) - \Hc(g) \leq \Hc(fg) \leq \Hc(f) + \Hc(g) $,
			\item $ \Hc(f^n) = \abs{n} \cdot \Hc(f) $,
			\item $ \Hc(f) = 0 \iff f \in \CC^* $,
			\item $ \Hc(A(f)) = \deg A \cdot \Hc(f) $ for any $ A \in \CC[T] \setminus \set{0} $.
		\end{enumerate}
	\end{mylemma}
	
	Furthermore, the following theorem due to Brownawell and Masser is an important ingredient in our proofs. It is an immediate consequence of Theorem B in \cite{brownawell-masser-1986}:
	
	\begin{mythm}[Brownawell-Masser]
		\label{thm:brownawellmasser}
		Let $ F/\CC $ be a function field in one variable of genus $ \gfr $. Moreover, for a finite set $ S $ of valuations, let $ u_1,\ldots,u_k $ be $ S $-units and
		\begin{equation*}
			1 + u_1 + \cdots + u_k = 0,
		\end{equation*}
		where no proper subsum of the left hand side vanishes. Then we have
		\begin{equation*}
			\max_{i=1,\ldots,k} \Hc(u_i) \leq \bino{k}{2} \left( \abs{S} + \max \setb{0, 2\gfr-2} \right).
		\end{equation*}
	\end{mythm}
	
	Finally, we need an auxiliary result about multiplicatively independent elements which is proven in \cite{fuchs-heintze-p8}:
	
	\begin{mylemma}
		\label{lemma:quotofindep}
		Let $ \gamma, \delta \in F \setminus \CC $ be multiplicatively independent and $ n,m \in \NN $. Assume that
		\begin{equation*}
			\Hc \left( \frac{\gamma^n}{\delta^m} \right) \leq L.
		\end{equation*}
		Then there exists an effectively computable constant $ C $, depending only on $ \gamma, \delta, \gfr $ and $ L $, such that
		\begin{equation*}
			\max \setb{n,m} \leq C.
		\end{equation*}
	\end{mylemma}
	
	\section{Proofs}
	
	Now all necessary preparations are finished and we start with proving our first theorem:
	
	\begin{proof}[Proof of Theorem \ref{thm:onesummand}]
		Without loss of generality we may assume that $ \alpha_1, \ldots, \alpha_d $ as well as $ f_1, \ldots, f_d $ are all $ S $-units since increasing the set $ S $ only makes the theorem stronger.
		Then the condition $ G_n \in \Oc_S^* $ can be written as the Diophantine equation
		\begin{equation}
			\label{eq:suniteqth1}
			f_1 \alpha_1^n + \cdots + f_d \alpha_d^n - s = 0
		\end{equation}
		in unknowns $ n \in \NN $ and $ s \in \Oc_S^* $.
		This is an $ S $-unit equation and therefore we aim for applying the inequality of Brownawell-Masser.
		
		As $ d \geq 2 $ and no summand is zero, there exists at least one minimal vanishing subsum, i.e. no proper subsubsum of this subsum vanishes, which contains two summands $ f_i \alpha_i^n $ and $ f_j \alpha_j^n $ for $ i \neq j $.
		Dividing this equation by $ f_j \alpha_j^n $, Theorem \ref{thm:brownawellmasser} gives
		\begin{equation*}
			\Hc \left( \frac{f_i \alpha_i^n}{f_j \alpha_j^n} \right) \leq \bino{d}{2} \left( \abs{S} + \max \setb{0, 2\gfr-2} \right) =: C_1.
		\end{equation*}
		
		From this we get by using some properties of the height given in Lemma \ref{lemma:heightproperties} the upper bound
		\begin{align*}
			n \cdot \Hc \left( \frac{\alpha_i}{\alpha_j} \right) &= \Hc \left( \frac{\alpha_i^n}{\alpha_j^n} \right) \\
			&\leq \Hc \left( \frac{f_i \alpha_i^n}{f_j \alpha_j^n} \right) + \Hc \left( \frac{f_j}{f_i} \right) \\
			&\leq C_1 + \max_{1 \leq k,l \leq d} \Hc \left( \frac{f_k}{f_l} \right) =: C_2
		\end{align*}
		and, since $ G_n $ is non-degenerate, finally
		\begin{equation*}
			n \leq \frac{C_2}{\Hc \left( \frac{\alpha_i}{\alpha_j} \right)} \leq \frac{C_2}{\min_{k \neq l} \Hc \left( \frac{\alpha_k}{\alpha_l} \right)} =: C_3.
		\end{equation*}
		Thus the theorem is proven.
	\end{proof}
	
	The proof of our second theorem follows the same idea but it becomes more complicated due to the second exponential variable:
	
	\begin{proof}[Proof of Theorem \ref{thm:twosummands}]
		As in the previous proof we may assume that $ \alpha_1, \ldots, \alpha_d $ as well as $ f_1, \ldots, f_d $ are all $ S $-units since increasing the set $ S $ only makes the theorem stronger.
		Then the condition $ G_n + G_m \in \Oc_S^* $ with $ n > m $ can be written as the Diophantine equation
		\begin{equation}
			\label{eq:suniteqth2}
			f_1 \alpha_1^n + \cdots + f_d \alpha_d^n + f_1 \alpha_1^m + \cdots + f_d \alpha_d^m - s = 0
		\end{equation}
		in unknowns $ n > m \in \NN $ and $ s \in \Oc_S^* $.
		This is an $ S $-unit equation and therefore we aim for applying the inequality of Brownawell-Masser.
		
		Now we distinguish some different cases.
		First assume that there is a minimal vanising subsum containing two summands $ f_i \alpha_i^m $ and $ f_j \alpha_j^m $ for $ i \neq j $.
		In the same way as in the proof of Theorem \ref{thm:onesummand} we get by using Theorem \ref{thm:brownawellmasser} the bound $ m \leq C_4 $.
		If there is also a minimal vanishing subsum containing two summands $ f_k \alpha_k^n $ and $ f_l \alpha_l^n $ for $ k \neq l $, then we get in the same way $ n \leq C_5 $ and are done.
		Otherwise there is a minimal vanishing subsum containing two summands $ f_k \alpha_k^n $ and $ f_l \alpha_l^m $.
		Dividing this equation by $ f_l \alpha_l^m $, Theorem \ref{thm:brownawellmasser} implies
		\begin{equation*}
			\Hc \left( \frac{f_k \alpha_k^n}{f_l \alpha_l^m} \right) \leq \bino{2d}{2} \left( \abs{S} + \max \setb{0, 2\gfr-2} \right) =: C_6.
		\end{equation*}
		Hence by Lemma \ref{lemma:heightproperties} we get
		\begin{align*}
			n \cdot \Hc (\alpha_k) &= \Hc (\alpha_k^n) \\
			&\leq \Hc \left( \frac{f_k \alpha_k^n}{f_l \alpha_l^m} \right) + \Hc (f_k) + \Hc (f_l) + \Hc (\alpha_l^m) \\
			&\leq C_6 + 2 \cdot \max_{1 \leq r \leq d} \Hc (f_r) + m \cdot \Hc (\alpha_l) \\
			&\leq C_6 + 2 \cdot \max_{1 \leq r \leq d} \Hc (f_r) + C_4 \cdot \max_{1 \leq r \leq d} \Hc (\alpha_r) =: C_7
		\end{align*}
		and thus
		\begin{equation*}
			n \leq \frac{C_7}{\Hc (\alpha_k)} \leq \frac{C_7}{\min_{1 \leq r \leq d} \Hc (\alpha_r)} =: C_8.
		\end{equation*}
		
		The second case that there is a minimal vanising subsum containing two summands $ f_i \alpha_i^n $ and $ f_j \alpha_j^n $ for $ i \neq j $ is completely analogous.
		
		In the sequel we may assume that each minimal vanishing subsum contains neither two summands with $ n $ nor two summands with $ m $ in the exponent.
		Therefore any minimal vanishing subsum contains exactly one term of the form $ f_i \alpha_i^n $ and one term of the form $ f_j \alpha_j^m $.
		
		In the case $ i = j $ Theorem \ref{thm:brownawellmasser} yields
		\begin{equation*}
			\Hc (\alpha_i^{n-m}) = \Hc \left( \frac{f_i \alpha_i^n}{f_j \alpha_j^m} \right) \leq C_6
		\end{equation*}
		which implies
		\begin{equation*}
			n-m \leq \frac{C_6}{\min_{1 \leq r \leq d} \Hc (\alpha_r)} =: C_9.
		\end{equation*}
		Thus we have the representation $ n = m + b $ with $ b \leq C_9 $.
		For each of this finitely many possiblities we insert the representation into eqaution \eqref{eq:suniteqth2} to get
		\begin{equation*}
			f_1 (1 + \alpha_1^b) \alpha_1^m + \cdots + f_d (1 + \alpha_d^b) \alpha_d^m - s = 0.
		\end{equation*}
		Now we enlarge the set $ S $ once to a set $ S' $ such that $ 1 + \alpha_r^b $ is an $ S' $-unit for all $ r = 1, \ldots, d $ and $ b \leq C_9 $.
		This yields
		\begin{equation}
			\label{eq:reduced}
			f_1' \alpha_1^m + \cdots + f_d' \alpha_d^m - s = 0
		\end{equation}
		for new coefficients $ f_1', \ldots, f_d' $.
		By Theorem \ref{thm:onesummand} we get $ m \leq C_{10} $.
		Hence we end with $ n \leq C_{10} + C_9 =: C_{11} $.
		
		In the final case $ i \neq j $ Theorem \ref{thm:brownawellmasser} gives
		\begin{equation*}
			\Hc \left( \frac{f_i \alpha_i^n}{f_j \alpha_j^m} \right) \leq C_6
		\end{equation*}
		and, by Lemma \ref{lemma:heightproperties}, it follows
		\begin{equation*}
			\Hc \left( \frac{\alpha_i^n}{\alpha_j^m} \right) \leq C_{12}.
		\end{equation*}
		Since the $ \alpha_r $ are not constant and pairwise multiplicatively independent, we can apply Lemma \ref{lemma:quotofindep} which gives us the bound $ \max \setb{n,m} \leq C_{13} $.
		Taking the maximum of all appearing upper bounds concludes the proof.
	\end{proof}

\end{document}